\documentclass[12pt, reqno]{amsart}

\usepackage{amsthm,amssymb,amstext,amscd,amsfonts,amsbsy,amsrefs,amsxtra,latexsym,amsmath,xcolor,mathrsfs,fancybox,upgreek, soul,url}
\usepackage[english]{babel}
\usepackage[all,cmtip]{xy}
\usepackage[latin1]{inputenc}
\usepackage{cancel}
\usepackage[draft]{hyperref}

\usepackage{comment}
\usepackage{mdframed}
\allowdisplaybreaks

\usepackage{mathtools}
\usepackage{enumerate}
\usepackage{thmtools}
\usepackage{thm-restate}
\usepackage{chngcntr}
\usepackage{enumitem}
\usepackage{etoolbox}
\usepackage{todonotes}

\usepackage{tikz}
\usepackage{verbatim}
\usetikzlibrary{quotes,angles}

\DeclarePairedDelimiter\abs{\lvert}{\rvert}
\DeclarePairedDelimiter\norm{\lVert}{\rVert}

\makeatletter
\let\oldabs\abs
\def\abs{\@ifstar{\oldabs}{\oldabs*}}
\let\oldnorm\norm
\def\norm{\@ifstar{\oldnorm}{\oldnorm*}}
\makeatother

\makeatletter
\glb@settings 
\makeatother

\newtheorem{theorem}{Theorem}

\newtheorem{corollary}[theorem]{Corollary}
\newtheorem{proposition}[theorem]{Proposition}

\theoremstyle{definition}

\theoremstyle{remark}

\newtheorem*{remarks}{Remarks}

\numberwithin{theorem}{section}
\numberwithin{proposition}{section}
\numberwithin{lemma}{section}
\numberwithin{corollary}{section}
\numberwithin{equation}{section}
\numberwithin{conjecture}{section}

\setlist[enumerate,1]{before=}
\AfterEndEnvironment{enumerate}{}

\newcommand{\Z}{\mathbb{Z}}
\newcommand{\R}{\mathbb{R}}
\newcommand{\C}{\mathbb{C}}

\newcommand{\re}{\textnormal{Re}}

\usepackage{bm}

\renewcommand{\pmod}[1]{\  \,  \left( \mathrm{mod} \,  #1 \right)}

\newcommand\blfootnote[1]{%
	\begingroup
	\renewcommand\thefootnote{}\footnote{#1}%
	\addtocounter{footnote}{-1}%
	\endgroup
}

\author{Andrew Baker$^\dagger$}
\author{Joshua Males}
\address{Department of Mathematics, Machray Hall, University of Manitoba, Winnipeg,
	Canada}
\email{bakera6@myumanitoba.ca}
\email{joshua.males@umanitoba.ca}

\begin{document}

\title[The BG-rank and $2$-quotient rank of partitions]{Asymptotics, Tur\'{a}n inequalities, and the distribution of the BG-rank and $2$-quotient rank of partitions}
\blfootnote{$^\dagger$ Undergraduate author at the University of Manitoba.}
\thanks{The research of the second author conducted for this paper is supported by the Pacific Institute for the Mathematical Sciences (PIMS). The research and findings may not reflect those of the Institute.}

\begin{abstract}
Let $j,n$ be even positive integers, and let $\overline{p}_j(n)$ denote the number of partitions with BG-rank $j$, and $\overline{p}_j(a,b;n)$ to be the number of partitions  with BG-rank $j$ and $2$-quotient rank congruent to $a \pmod{b}$. We give asymptotics for both statistics, and show that $\overline{p}_j(a,b;n)$ is asymptotically equidistributed over the congruence classes modulo $b$. We also show that each of $\overline{p}_j(n)$ and $\overline{p}_j(a,b;n)$ asymptotically satisfy all higher-order Tur\'{a}n inequalities.
\end{abstract}

\maketitle

\section{Introduction and statement of results}

A partition $\lambda = (\lambda_1,\dots,\lambda_{s})$ of a positive integer $n$ is a list of non-increasing parts $\lambda_j$ such that $\lvert \lambda \rvert \coloneqq \sum_{j=1}^{s} \lambda_j=n$. Letting $p(n)$ denote the number of partitions of $n$, Hardy and Ramanujan \cite{HardyRamanujan} proved their asymptotic formula 
\begin{equation}\label{HR}
	p(n)\sim \frac{1}{4\sqrt{3}n} e^{\pi \sqrt{\frac{2n}{3}}},
\end{equation}
as $n\rightarrow \infty.$ This work marked the birth of the powerful Circle Method that has seen a wide range of applications in the last century, in particular in obtaining asymptotic and exact formulae for partitions and their statistics. For example, the exact formula for $p(n)$ given by Rademacher \cite{RademacherExact}, and formulae and inequalities for certain partition statistics \cite{Bring,BM} among many others.

We use Wright's variant of the Circle Method to obtain asymptotic formulae and distribution results for a certain family of partitions, first introduced by Berkovich and Garvan in \cite{BG}. For this we need two partition statistics. Recall that a partition has a Ferrers--Young diagram given by
$$
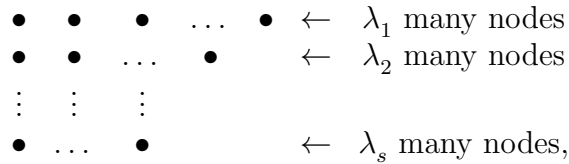
\begin{matrix}
	\bullet & \bullet & \bullet & \dots & \bullet & \leftarrow & \lambda_1 \text{ many nodes}\\
	\bullet & \bullet & \dots & \bullet & & \leftarrow & \lambda_2 \text{ many nodes}\\
	\vdots & \vdots & \vdots & &  &  \\
	\bullet & \dots & \bullet & & & \leftarrow & \lambda_s \text{ many nodes},
\end{matrix}
$$
and each node has a hook length. The node in row $k$ and column $j$ has hook length
$h(k,j):=(\lambda_k-k)+(\lambda'_j-j)+1,$ where $\lambda'_j$ is the number of nodes in column $j$. A $t$-core partition is a partition that has no hook lengths divisible by $t$. Let $P$ denote the set of all partitions, $P_t$ the set of $t$-core partitions, and consider the well-known map of Littlewood \cite{Littlewood} $\phi_1 \colon P \to P_t \times P \times \dots \times P$ defined via
\begin{align*}
	\phi_1(\lambda) = (\lambda_{\text{$t$-core}}, \hat{\lambda}_t),\qquad \hat{\lambda}_t = (\hat{\lambda}_0,\hat{\lambda}_1,\dots,\hat{\lambda}_{t-1})
\end{align*}
such that
\begin{align*}
	|\lambda| = |\lambda_{\text{$t$-core}}| +t \sum_{j=0}^{t-1} |\hat{\lambda}_j|.
\end{align*}
Then the $2$-quotient rank of $\lambda$ is defined by \cite{BG}
\begin{align*}
	\text{$2$-quotient rank}(\lambda) \coloneqq \nu(\lambda_0) - \nu(\lambda_1),
\end{align*}
where $\nu(\lambda)$ is the number of parts of $\lambda$ and
\begin{align*}
	\phi_1(\lambda) = (\lambda_{\text{$2$-core}},(\hat{\lambda}_0,\hat{\lambda}_1)).
\end{align*}
It was shown in \cite[Theorem 3.1]{BG} that the $2$-quotient rank naturally divides partitions of $5n+4$ with so-called s-rank $i=0,2$ into five equal classes, thereby proving an analogue of Ramanujan's congruence $p(5n+4) \equiv 0 \pmod{5}$ in these cases.
The BG-rank was also defined in \cite{BG}
\begin{align*}
	\text{BG-rank}(\lambda) \coloneqq \sum_{j=1}^{\nu} (-1)^{j+1} \operatorname{par}(\lambda_j),
\end{align*}
where $\operatorname{par}$ denotes the parity of a positive integer. It was used to prove a further refinement of Ramanujan's congruence modulo $5$ as well as to prove new congruences modulo $5$ of $p(n)$.

Let $\overline{p}_j(m,n)$ denote the number of partitions of a positive integer $n$ with BG-rank $j$ and $2$-quotient rank $m$. Then Berkovich and Garvan \cite{BG} proved that the generating function is
\begin{align*}
	H(\zeta;q) \coloneqq \sum_{\substack{n \geq 0 \\ m \in \Z}} \overline{p}_j(m,n) \zeta^m q^n = \frac{q^{(2j-1)j}}{(q^2 \zeta,q^2 \zeta^{-1};q^2)_\infty},
\end{align*}
where $(a;q)_\infty \coloneqq \prod_{n\geq0} (1-aq^n)$ is the usual $q$-Pochhammer symbol, and we define $(a,b;q)_\infty \coloneqq (a;q)_\infty(b;q)_\infty$.

Let $\overline{p}_j(n)$ count the number of partitions of $n$ with BG-rank $j$, and $\overline{p}_j(a,b;n)$ denote the number of partitions with BG-rank $j$ and $2$-quotient rank congruent to $a \pmod{b}$. Then a standard computation using orthogonality of roots of unity gives
\begin{align}\label{eqn: splitting}
H(a,b;q) \coloneqq	\sum_{n \geq 0} \overline{p}_j(a,b;n) q^n = \frac{1}{b} \sum_{n \geq 0} \overline{p}_j(n) q^n + \frac{1}{b} \sum_{k=1}^{b-1} \zeta_b^{-ak} 	H(\zeta_b^k;q).
\end{align}

We follow the framework provided by \cite{CCM}, which builds on work of \cite{BCMO}, and prove the following equidistribution result. Note that a similar result may be obtained for odd $j$.

\begin{theorem}\label{thm: main}
	Let $0\leq a <b$ and $b>1$. Assume that $j$ and $n$ are even. Then as $n \to \infty$ we have
	\begin{align*}
		\overline{p}_j(a,b;n) = \frac{2}{3^{\frac{3}{4}} (2n)^{\frac{5}{4}} b } e^{\pi \sqrt{\frac{2n}{3}}} \left(1+ O \left(n^{-\frac{1}{2}}\right)\right).
	\end{align*}

\end{theorem}

A direct corollary yields the asymptotic behaviour of $\overline{p}_j(n)$.
\begin{corollary}\label{cor: asymps of p_j}
	Let $j$ and $n$ be even. As $n \to \infty$ we have
	\begin{align*}
		\overline{p}_j(n) = \frac{2}{3^{\frac{3}{4}} (2n)^{\frac{5}{4}} } e^{\pi \sqrt{\frac{2n}{3}}} \left(1+ O \left(n^{-\frac{1}{2}}\right)\right).
	\end{align*}

\end{corollary}

\begin{remarks}
	\begin{enumerate}[label=(\roman*)]
		\item We remark that neither Theorem \ref{thm: main} or Corollary \ref{cor: asymps of p_j} depend on $j$, implying that asymptotically the BG-rank is also equidistributed over partitions. However, the dependence on $j$ can be seen if one applies further terms arising from the application of Wright's Circle Method in the proof of Theorem \ref{thm: main}.
		
		\item For $b=5$ we recover an asymptotic form of \cite[Theorem 7.1]{BG}.
	\end{enumerate} 
\end{remarks}

A direct application of \cite[Corollary 3.2]{CCM}  to our results gives that both $\overline{p}_j(n)$ and $\overline{p}_j(a,b;n)$ asymptotically satisfy the following convexity results. Similar results for other partition-theoretic objects have been obtained by many authors, see e.g. \cite{BO,HJ,Males}.
\begin{corollary}\label{cor: convexity}
	Let $j$ be even. For large enough even $n_1$ and $n_2$, we have that
	\begin{align*}
		\overline{p}_j(a,b;n_1)\overline{p}_j(a,b;n_2) > \overline{p}_j(a,b;n_1+n_2),
	\end{align*}
and 
\begin{align*}
	\overline{p}_j(n_1) \overline{p}_j(n_2) > \overline{p}_j(n_1+n_2).
	\end{align*}
\end{corollary}

The Tur\'{a}n inequalities for functions in the real entire Laguerre--P\'{o}lya class are intricately linked to the Riemann hypothesis \cite{Dimitrov, Sz} and have seen renewed interest in recent years. In fact, the Riemann hypothesis is true if and only if the Riemann Xi function lies in the Laguerre--P\'{o}lya class, and a necessary condition is that the Maclaurin coefficients satisfy Tur\'{a}n inequalities of all orders \cite{Dimitrov,SchurPolya}.

The second order Tur\'{a}n inequality for a sequence $\{a_n\}_{n \geq 0}$ is known as log-concavity, and is satisfied if
\begin{align*}
	a_n^2 \geq a_{n-1}a_{n+1}
\end{align*}
for all $n \geq 1$. Log-concavity for modular type objects is well-studied in the literature, for example \cite{BJMR,CCM,CraigPun,DawMas,DeSalvo}. For an explicit description of the higher order inequalities, we refer the reader to e.g. the introduction of \cite{CraigPun}. 
The higher-order Tur\'{a}n inequalities are linked to the Jensen polynomial associated to a sequence $\{\alpha(n)\}$, given by
\begin{align*}
	J_{\alpha}^{d,n}(X) \coloneqq \sum_{k=0}^{d-1} \binom{d}{k} \alpha(n+k) X^k.
\end{align*}
In particular, all higher-order Tur\'{a}n inequalities are satisfied by the sequence $\alpha(n)$ if the associated Jensen polynomial is hyperbolic. Similar situations to ours have been studied in e.g. \cite{GORZ,LW,OPR}.

\begin{theorem}\label{thm: turan ineqs}
	Let $j$ be even and let  $\widehat{J}_\alpha^{d,2n}$ be the renormalised Jensen polynomials defined in Theorem \ref{Thm: Gorz} and $H_d(X)$ the Hermite polynomials. Then we have
	\begin{align*}
	 \lim_{n \to \infty} \widehat{J}_{\overline{p}_j}^{d,2n}(X) = H_d(X)
	\end{align*}
uniformly for $X$ on compact subsets of $\R$. This implies that both $\overline{p}_j(a,b;2n)$ and $\overline{p}_j(2n)$ asymptotically satisfy all higher order Tur\'{a}n inequalities.
\end{theorem}

\subsection*{Outline}
We begin in Section \ref{Sec: prelims} by recalling preliminary results needed through the rest of the paper. Section \ref{Sec: proofs} is dedicated to proving Theorem \ref{thm: main} and Corollary \ref{cor: asymps of p_j}. In Section \ref{Sec: Turan} we prove Theorem \ref{thm: turan ineqs}.

\subsection*{Acknowledgements}
The authors thank the Manitoba eXperimental Mathematics Laboratory for their support of the project. 

\section{Preliminaries}\label{Sec: prelims}

\subsection{Asymptotics of infinite $q$-products}
Following \cite{BCMO}, we let
\begin{align*}
	F_1(\zeta; q)\coloneqq \prod_{n=1}^{\infty}\left(1-\zeta q^n\right), \qquad 
\end{align*}
along with Lerch's transcendent
\begin{align*}
	\Phi(z,s,a)\coloneqq \sum_{n=0}^\infty \frac{z^n}{(n+a)^s},
\end{align*}
and for $0\leq \theta < \frac{\pi}{2}$ define the domain $D_{\theta} \coloneqq \left\{ z=re^{i\alpha} \colon r \geq 0 \text{ and } |\alpha| \leq \theta \right\}$. Throughout, the Gamma function is defined as usual by $\Gamma(x)\coloneqq \int_0^\infty t^{x-1} e^{-t} dt $, for $\operatorname{Re}(x)>0$.
Then \cite[Theorem 2.1]{BCMO} in our context reads as follows.
\begin{theorem} \label{Thm: asymptotics F}
	For $b\geq 2$, let $\zeta$ be a primitive $b$-th root of unity. Then as $z \to 0$ in $D_\theta$, we have 
		\begin{align*}
			F_{1}\left(\zeta;e^{-z}\right)  =\frac{1}{\sqrt{1-\zeta}} \, e^{-\frac{\zeta\Phi(\zeta,2,1)}{z}}\left( 1+O\left(|z|\right) \right).
		\end{align*}
\end{theorem}
In particular, this allows us to obtain estimates of $F_1$ on the major arcs arising from Wright's Circle Method.

\subsection{Wright's Circle Method}
Our central tool is a variant of Wright's Circle Method, which was proved by Bringmann, Craig, Ono, and one of the authors \cite[Proposition 4.4]{BCMO}, following work of Wright \cite{Wright2} - see also Ngo and Rhoades \cite{NgoRhoades}. The essence of Wright's Circle Method is that one uses Cauchy's integral formula to recover the the Fourier coefficients as an integral of the generating function over a complex circle transversed exactly once in the counterclockwise direction. One then splits the circle into major (resp.\@ minor) arcs, where the generating function has relatively large (resp.\@ small) growth. In turn, this provides one with an asymptotic estimate for the coefficients. Often, the major arc is close to $q=1$. In the present paper, we also have a major arc near $q=-1$ owing to the factor of $(q^2;q^2)_\infty^{-2}$ in the first term of the right-hand side of \eqref{eqn: splitting}. 

In \cite{BCMO} the following result is proved when the dominant pole is near $q=1$. We will follow the proof of this result, with very minor modifications to account for a second major arc near $q=-1$. In essence, the major arc near $q=-1$ will mean that we include an extra factor of $2$ in the asymptotic form of the coefficients from the following Proposition.

\begin{proposition} \label{WrightCircleMethod}
	Suppose that $F(q)$ is analytic for $q = e^{-z}$ where $z=x+iy \in \C$ satisfies $x > 0$ and $|y| < \pi$, and suppose that $F(q)$ has an expansion $F(q) = \sum_{n=0}^\infty c(n) q^n$ near 1. Let $N,M>0$ be fixed constants. Consider the following hypotheses:
	
	\begin{enumerate}[leftmargin=*]
		\item[\rm(1)] As $z\to 0$ in the bounded cone $|y|\le Mx$ (major arc), we have
		\begin{align*}
			F(e^{-z}) = z^{B} e^{\frac{A}{z}} \left( \sum_{j=0}^{N-1} \alpha_j z^j + O_M\left(|z|^N\right) \right),
		\end{align*}
		where $\alpha_j \in \C$, $A\in \R^+$, and $B \in \R$. 
		
		\item[\rm(2)] As $z\to0$ in the bounded cone $Mx\le|y| < \pi$ (minor arc), we have 
		\begin{align*}
			\lvert	F(e^{-z}) \rvert \ll_M e^{\frac{1}{\mathrm{Re}(z)}(A - \kappa)},
		\end{align*}
		for some $\kappa\in \R^+$. 
	\end{enumerate}
	If  {\rm(1)} and {\rm(2)} hold, then as $n \to \infty$ we have for any $N\in \R^+$ 
	\begin{align*}
		c(n) = n^{\frac{1}{4}(- 2B -3)}e^{2\sqrt{An}} \left( \sum\limits_{r=0}^{N-1} p_r n^{-\frac{r}{2}} + O\left(n^{-\frac N2}\right) \right),
	\end{align*}
	where $p_r \coloneqq  \sum\limits_{j=0}^r \alpha_j c_{j,r-j}$ and $c_{j,r} \coloneqq  \dfrac{(-\frac{1}{4\sqrt{A}})^r \sqrt{A}^{j + B + \frac 12}}{2\sqrt{\pi}} \dfrac{\Gamma(j + B + \frac 32 + r)}{r! \Gamma(j + B + \frac 32 - r)}$. 
\end{proposition}

\subsection{Hyperbolicity of Jensen polynomials}

We use a result of \cite{GORZ}, where the authors showed that for certain classes of functions, the Jensen polynomials of the coefficients tend to the order $d$ Hermite polynomials $H_d(X)$, in turn implying the asymptotic hyperbolicity of the Jensen polynomials. 

\begin{theorem}[Theorem 3 of \cite{GORZ}]\label{Thm: Gorz}
	Let $\{ \alpha(n)\}$, $\{A(n)\}$, and $\{\delta(n)\}$ be three sequences of positive real numbers with $\delta(n)$ tending to $0$ and satisfying
	\begin{align*}
		\log\left( \frac{\alpha(n+k)}{\alpha(n)} \right) = A(n)k - \delta(n)^2k^2+\sum_{i=3}^{d} g_i(n)k^i +o(\delta(n)^d)
	\end{align*}
	as $n \to \infty$, where each $g_i(n) = o(\delta(n)^i)$ for each $3\leq i \leq d$. Then the renormalized Jensen polynomials 
	\begin{align*}
		\widehat{J}_{\alpha}^{d,n}(X) \coloneqq \frac{\delta(n)^{-d}}{\alpha(n)} J_{\alpha}^{d,n}\left( \frac{\delta(n)X-1}{\exp(A(n))}\right)
	\end{align*}
	satisfy
	\begin{align*}
		\lim_{n \to \infty} \widehat{J}_{\alpha}^{d,n}(X) = H_d(X),
	\end{align*}
	uniformly for $X$ in any compact subset of $\mathbb{R}$. Moreover, this implies that the Jensen polynomials $J_{\alpha}^{d,n}$ are each hyperbolic for all but finitely many $n$.

\end{theorem}

\section{Proofs of Theorem \ref{thm: main} and Corollary \ref{cor: asymps of p_j}}\label{Sec: proofs}
In this section we prove the main results on the asymptotic behaviour of $\overline{p}_j(a,b;n)$ and $\overline{p}_j(n)$.

\begin{proof}[Proof of Theorem \ref{thm: main}]
	We begin by rewriting
	\begin{align*}
		H(\zeta;q) = \frac{q^{(2j-1)j}}{F_1(\zeta;q^2) F_1(\zeta^{-1};q^2)},
	\end{align*}
noting that $H(1;q) = \frac{q^{(2j-1)j}}{(q^2;q^2)_\infty^2}$. Then by \eqref{eqn: splitting} we obtain
\begin{align*}
	H(a,b;q)  = \frac{q^{(2j-1)j}}{b (q^2;q^2)_\infty^2} +  \frac{1}{b} \sum_{k=1}^{b-1} \zeta_b^{-ak} 	 \frac{q^{(2j-1)j}}{F_1(\zeta_b^j;q^2) F_1(\zeta_b^{-j};q^2)}.
\end{align*}
We aim to prove that the first term on the right-hand side dominates on both the major and minor arcs, and that the resulting asymptotic expansion for $q=e^{-z}$ as $z \to 0$ satisfies the analogue of Proposition \ref{WrightCircleMethod} with a second major arc at $q=-1$ contributing as well.

First consider the major arc near $q=1$, so $|y|\le Mx$ for arbitrary $M>0$. Letting $P(q) \coloneqq \sum_{n \geq 0} p(n)q^n = (q;q)_\infty^{-1}$, we have the classical asymptotic (for $|y|\le Mx$, as $z\to0$) - see e.g. see 5.8.1 of \cite{CS},
\[
P\left(e^{-z}\right) = \sqrt{\frac{z}{2\pi}} e^{\frac{\pi^2}{6z}} (1+O(|z|))
\]
and so
\begin{align*}
 e^{-(2j-1)jz} \left(e^{-2z}; e^{-2z}\right)_\infty^{-2} =  \frac{z}{\pi} e^{\frac{\pi^2}{6z}} (1+O(|z|)),
\end{align*}
where we expanded the first exponential using its Taylor expansion. Note that for $q=-e^{-z}$ near to $q=-1$ the same asymptotic holds, using that $j$ is even.
By Theorem \ref{Thm: asymptotics F}, for $\zeta_b^j \neq 1$ we have that
\begin{align*}
		& F_{1}\left(\zeta_b^j;e^{-2z}\right)  = \frac{1}{\sqrt{1-\zeta_b^j}} \, e^{-\frac{\zeta_b^j\Phi(\zeta_b^j,2,1)}{2z}}\left( 1+O\left(|z|\right) \right),\\
		&F_{1}\left(\zeta_b^{-j};e^{-2z}\right)  =\frac{1}{\sqrt{1-\zeta_b^{-j}}} \, e^{-\frac{\zeta_b^{-j}\Phi(\zeta_b^{-j},2,1)}{2z}}\left( 1+O\left(|z|\right) \right).
\end{align*}
So the major arc bound follows if 
	\begin{align*}
		\frac{1}{\sqrt{(1-\zeta_b^j)(1-\zeta_b^{-j}) }} \, e^{-\frac{\zeta_b^j\Phi(\zeta_b^j,2,1) +\zeta_b^{-j}\Phi(\zeta_b^{-j},2,1)}{2z}}  \left( 1+O\left(|z|\right) \right) \ll \sqrt{\frac{z}{2\pi}} e^{\frac{\pi^2}{6z}} (1+O(|z|)).
	\end{align*}
Since $\zeta \Phi(\zeta,2,1) = \operatorname{Li}_2(\zeta)$ is the dilogarithm function, and $\operatorname{Li}_2(z) + \operatorname{Li}_2(z^{-1}) = -\frac{\pi^2}{6} -\frac{1}{2}\log(-z)^2$, the bound follows if
\begin{align*}
\re \left(	\frac{\pi^2}{12z} + \frac{\log(-\zeta_b^j)^2}{4z} \right) \ll \re \left( \frac{\pi^2}{6z} \right).
\end{align*}
Note that $\log(-\zeta_b^j) = \log(1)+i \arg(-\zeta_b^j) = i \arg(-\zeta_b^j)$, so the condition becomes
\begin{align*}
	\pi^2 - 3\arg(-\zeta_b^j)^2 < 2\pi^2,
\end{align*}
which is clearly satisfied. Thus on the major arcs as $z \to 0$, we obtain that
\begin{align*}
	H(a,b;\pm e^{-z}) =  \frac{z}{b \pi} e^{\frac{\pi^2}{6z}} (1+O(|z|)).
\end{align*}

We now turn to the minor arcs, where $y \geq Mx$. It is well-known that for some $\mathcal{C}>0$ (see e.g. \cite[Lemma 3.5]{BD})
\[
\left|P\left(e^{-z}\right)\right| \le x^{\frac12} e^{\frac{\pi}{6x}-\frac{\mathcal{C}}{x}},
\]
and so 
\begin{align*}
	\left\lvert e^{-(2j-1)jz} \left(e^{-2z}; e^{-2z}\right)_\infty^{-2} \right\rvert \le 2x e^{\frac{\pi}{6x}-\frac{\mathcal{C}}{x}}.
\end{align*}
Next we turn to the function $\frac{1}{F_1(\zeta;q)}$ for a root of unity $\zeta \neq 1$. We have
\begin{align*}
	\log \left(F_1(\zeta;q)^{-1} \right) = \sum_{n \geq 1} -\log\left(1-\zeta q^n\right) = \sum_{n,m \geq 1} \frac{\zeta^m q^{mn}}{m} = \sum_{m \geq 1} \frac{\zeta^m q^m}{m(1-q^m)}.
\end{align*}
We then bound
\begin{align*}
	\log \left(F_1(\zeta;q)^{-1} \right) \leq \left\lvert \frac{\zeta q}{1-q} \right\rvert - \frac{\lvert q \rvert}{1-\lvert q \rvert} + \log P(\lvert q\rvert) 
\end{align*}
As in the proof of Theorem 1.4 of \cite{BCMO}, we obtain that 
\begin{align*}
	\left\lvert \frac{\zeta q}{1-q} \right\rvert - \frac{\lvert q \rvert}{1-\lvert q \rvert} = -\frac{1}{x} +O(1).
\end{align*}
Thus on the minor arc we have 
\begin{align*}
	F_1(\zeta_b^j;e^{-2z})^{-1} F_1(\zeta_b^{-j};e^{-2z})^{-1} \ll \left\lvert e^{-(2j-1)jz} \left(e^{-2z}; e^{-2z}\right)_\infty^{-2} \right\rvert \le 2x e^{\frac{\pi}{6x}-\frac{\mathcal{C}}{x}}.
\end{align*}
Finally, we we follow the proof of Proposition \ref{WrightCircleMethod} to obtain the result. For the major arc near $q=1$, the contribution is precisely that stated in Proposition \ref{WrightCircleMethod} with $B=1$, $N=1$, $A = \frac{\pi^2}{6}$ and $\alpha_0 = \frac{1}{b \pi}$. For the major arc near $q=-1$, the proof of \cite{BCMO} applies {\it mutatis mutandis}; if we label the arc by $C_{-1}$ and let
\begin{align*}
	A_0(n) \coloneqq \frac{1}{2 \pi i} \int_{C_{-1}} \frac{z^{B} e^{\frac{A}{z}}}{q^{n+1}} dq,
\end{align*}
then one may use the same estimates as in \cite{BCMO}. In particular, the $A_0$ term again dominates the asymptotic contribution on the arc $C_{-1}$. Writing $q= -e^{-z}$, since $n$ is even we see that this term will combine with the contributions from the major arc near $q=1$, which we label $C_1$. Then letting $z\mapsto -z$ we have the integral (noting that $B=1$)
\begin{align}\label{eqn: second contribution}
	\frac{1}{2 \pi i} \int_{C_{1}} z e^{-\frac{A}{z} -nz} dz.
\end{align}
The proof of \cite[Proposition 4.4]{BCMO} in turn relies on an estimate of an I-Bessel function in \cite[Lemma 3.7]{NgoRhoades}. One may estimate the integral in \eqref{eqn: second contribution} in exactly the same way, using the asymptotic formula \cite[10.30.5]{nist}. We therefore obtain an extra factor of $2$ for the two major arcs that contribute, and may use the asymptotic form of the coefficients stated in Proposition \ref{WrightCircleMethod} with this factor of $2$ included.
\end{proof}

We now prove Corollary \ref{cor: asymps of p_j}, relying on the proof of Theorem \ref{thm: main}.
\begin{proof}[Proof of Corollary \ref{cor: asymps of p_j}]
	The result follows by noting that the generating function for $\overline{p}_j(n)$ dominated on both the major and minor arc in the proof of Theorem \ref{thm: main}. Thus, up to a factor of $b$, the asymptotic from Theorem \ref{thm: main} holds, giving the result.
\end{proof}

\section{Tur\'{a}n Inequalities}\label{Sec: Turan}
This section is dedicated to proving Theorem \ref{thm: turan ineqs}. We directly apply Theorem \ref{Thm: Gorz} to the asymptotic formulae obtained in the previous section.

\begin{proof}[Proof of Theorem \ref{thm: turan ineqs}]
	We present the proof for $\overline{p}_j(n)$, as the proof for $\overline{p}_j(a,b;n)$ differs only by a factor of $b$. Applying Proposition \ref{WrightCircleMethod} but adjusting to use the full asymptotic for the $I$-Bessel functions occurring in its proof, it is easy to show that there are constants $c_\nu$ such that as $n \to \infty$
	\begin{align*}
				\overline{p}_j(n) = \frac{2}{3^{\frac{3}{4}} (2n)^{\frac{5}{4}} } \exp\left({\pi \sqrt{\frac{2n}{3}}}\right) \exp\left( c_0 + \frac{c_1}{n^{\frac{1}{2}}} + \frac{c_2}{n} +\dots \right)
	\end{align*}
to all orders of $n$.

	It follows in a similar way to \cite[Proof of Theorem 7]{GORZ} that
	\begin{align*}
		\log\left( \frac{\overline{p}_j(n+k)}{\overline{p}_j(n)}\right) \sim \pi \sqrt{\frac{2}{3}} \sum_{i \geq 1} \binom{1/2}{i} \frac{k^i}{n^{i-\frac{1}{2}}} - \frac{5}{4} \sum_{i \geq 1} \frac{(-1)^{i-1}k^i}{in^i} + \sum_{s,t\geq 1} c_t \binom{-t}{s} \frac{k^s}{n^{\frac{s+t}{2}}}.
	\end{align*}
It is not difficult to see that one can then apply Theorem \ref{Thm: Gorz} with the sequences $A(n)$, $\delta(n)$ defined via
\begin{align*}
	A(n) = \pi \sqrt{\frac{1}{6n}} + O\left(\frac{1}{n} \right), \qquad \delta(n)^2 = \pi \sqrt{\frac{2}{3}} \binom{1/2}{2} n^{-\frac{3}{2}} + O\left(n^{-2}\right). 
\end{align*}to obtain the result.
\end{proof}
We remark that it is not difficult to see that any function satisfying Proposition \ref{WrightCircleMethod} will have an associated Jensen polynomial that is asymptotically hyperbolic. We pose the following question for future research: can one prove lower bounds $N_{A,B}(d)$ such that the Jensen polynomial associated to such a function is hyperbolic for all $n>N_{A,B}(d)$? A resolution would immediately yield effective results for large swathes of partition theoretic objects, coefficients of modular-type objects, and topological invariants.

Moreover, since the Jensen polynomial attached to the sequences $\overline{p}_j(n)$ and $\overline{p}_j(a,b;n)$ for even $n$ are each eventually hyperbolic, the generating functions 
\begin{align*}
	\sum_{n \geq 0} \frac{\overline{p}_j(a,b;2n)}{n!} x^n, \qquad \sum_{n \geq 0} \frac{\overline{p}_j(2n)}{n!} x^n
\end{align*}
can be considered as new examples of functions in the shifted Laguerre--P\'{o}lya class, as in recent work of Wagner \cite{Wagner}. Since the main asymptotic term in each case arises from a weakly holomorphic modular form with at most a pole at infinity, this follows from the fact that by \cite{GORZ} all such functions lie in the shifted Laguerre--P\'{o}lya class. However, the coefficients $\overline{p}_j(a,b;2n)$ themselves are not coefficients of a such a modular form, and so provide at least a slightly more general object lying in this new class of object.

\subsection*{Declarations}
The authors have no competing interests to declare that are relevant to the content of this article.

\end{document}